\documentclass{article}
\usepackage{nameref, hyperref, cleveref, amsmath, amsthm, thmtools, thm-restate, amsfonts, amssymb, tikz, tikz-cd, xcolor, listings, pgf}
\usepackage[margin=3.5cm]{geometry}
\usetikzlibrary{arrows,shapes,decorations.pathreplacing}

\declaretheorem{theorem}
\declaretheorem[sibling=theorem, refname={conjecture, conjectures}, Refname={Conjecture, Conjectures}]{conjecture}
\declaretheorem[sibling=theorem]{lemma}
\declaretheorem[style=remark]{remark}
\declaretheorem[style=remark]{example}

\definecolor{pink1}{RGB}{219, 31, 175}
\definecolor{pink2}{RGB}{145, 31, 129}
\definecolor{pink3}{RGB}{84, 51, 76}

\title{The Hanna Neumann Conjecture and the rank of the join}
\author{Joshua E. Hunt\footnote{E-mail: \texttt{joshuahunt94@gmail.com} \newline
The author was funded by a Trinity College Summer Research Studentship.}}

\begin{document}
\maketitle

\abstract{The Hanna Neumann conjecture gives a bound on the intersection of finitely generated subgroups of free groups. We explore a natural extension of this result, which turns out to be true only in the finite index case, and provide counterexamples for the general case. We also see that the graph-based method of generating random subgroups of free groups developed by Bassino, Nicaud and Weil is well-suited to generating subgroups with non-trivial intersections. The same method is then used to generate a counterexample to a similar conjecture of Guzman.}

\section{Introduction}

Throughout, let $F$ be a finitely generated (non-trivial) free group, and $H, K$ be finitely generated subgroups of $F$. We define the \emph{reduced rank} of $H \leq F$ to be
\begin{equation*}
\overline r(H) = \max(0, \text{rank}(H)-1)
\end{equation*}
The Hanna Neumann conjecture states that
\begin{equation*}
\overline r(H \cap K) \leq \overline r(H) \, \overline r(K)
\end{equation*}
A strengthening of this, the Strengthened Hanna Neumann conjecture, was proposed by Walter Neumann \cite{Neu}:
\begin{equation*}
\sum_{HgK\text{ s.t. }H \cap gKg^{-1} \neq \{1\}} \overline r(H \cap g K g^{-1}) \leq \overline r(H) \, \overline r(K)
\end{equation*}
This was proved independently by Joel Friedman \cite{Fri} and Igor Mineyev \cite{Min} in 2011.

The results above do not involve the join of $H$ and $K$ (\emph{i.e.} the group generated by $H \cup K$, which we denote $H \vee K$). By way of analogy with the inclusion/exclusion principle, it seems natural to suppose that when $H \cap K$ is ``large'', $H \vee K$ is ``not much bigger'' than $H$ or $K$. Accordingly, Henry Wilton conjectured\footnote{In a private communication}

\begin{conjecture}[Wilton]
Let $H, K$ be finitely generated subgroups of $F$. Then
\begin{equation*}
\overline r (H \vee K) \, \overline r(H \cap K) \leq \overline r(H) \, \overline r(K)
\end{equation*}
\end{conjecture}

We will refer to this as the Inclusion/Exclusion Hanna Neumann Conjecture (IEHNC). This turns out to be true if we also assume that $K$ is of finite index in $H \vee K$:

\begin{restatable}{theorem}{iehnc}
\label{theorem:IEHNC}
Let $H, K$ be finitely generated subgroups of $F$, with $K$ of finite index in $F$. Then
\begin{equation*}
\overline r (H \vee K) \, \overline r(H \cap K) \leq \overline r(H) \, \overline r(K)
\end{equation*}
\end{restatable}

Unfortunately, the IEHNC is no longer true if we do not make this assumption, and so does not give a strengthening of the Hanna Neumann conjecture.

Note that the Strong Hanna Neumann conjecture is invariant under conjugation of $H$ or $K$ in $F$, while the IEHNC is not. As such, we would not expect to be able to combine the two results.

The IEHNC is slightly stronger than Rosemary Guzman's ``Group-theoretic conjecture'' \cite{Guz}:

\begin{conjecture}[Guzman]
\label{conjecture:guzman}
Let $H, K$ be finitely generated subgroups of $F$ with
\begin{equation*}
m = \text{rk}(H) = \text{rk}(K) \leq \text{rk}(H \cap K)
\end{equation*}
for some $m \geq 2$. Then
\begin{equation*}
\text{rk}(H \vee K) \leq m
\end{equation*}
\end{conjecture}

Note that the IEHNC implies Guzman's conjecture. We give a counterexample to both Guzman's conjecture and the IEHNC (without the finite index assumption) in \cref{section:counterexamples}. Furthermore, \cref{example:IEHNC counterexample} in fact shows that there is no $c > 0$ such that
\begin{equation*}
\overline r (H \vee K) \, \overline r(H \cap K) \leq c \; \overline r(H) \, \overline r(K)
\end{equation*}
for all $H, K$ finitely generated subgroups of $F$.

The counterexample given to Guzman's conjecture has $m = 5$. Louder and McReynolds \cite{Lou} and Kent \cite{Ken} both proved conjecture \ref{conjecture:guzman} for the case $m=2$, and Guzman herself proved it for $m=3$ \cite{Guz}, so it remains to determine whether or not it holds for $m = 4$.

The first counterexamples to both the IEHNC and Guzman's conjecture were generated using a computer search, with the software package GAP. In \cref{section:graph generation} we discuss the methods used for this, and their suitability to investigating similar problems about free groups. We would be more than happy to share the source code on request.

\subsection*{Acknowledgements}

The author gratefully acknowledges the support of a Trinity College Summer Research Studentship, and would like to thank Henry Wilton for supervising the project. He would also like to thank Trenton Schirmer for pointing out the connection between the IEHNC and Guzman's conjecture.

\section{Stallings graphs}

\subsection{The category of graphs}
Finitely generated subgroups of free groups can be represented by immersions (locally injective maps) of finite graphs, as introduced by Stallings in \cite{Sta}. We consider a graph to be a pair $X = (V(X), E(X))$ of sets of vertices and directed edges, along with a function $E \to E$ and a function $E \to V$. A graph is finite if both its vertex set and edge set are finite. For each $e \in E$ we have an associated edge $\overline e \in E$, the reversal of $e$, and an associated vertex $\iota(e) \in V$, the initial vertex of $e$. We require that $\overline {\overline e} = e$ and $\overline e \neq e$. We define the terminal vertex of $e$ to be $\tau(e) = \iota(\overline e)$.

A map of finite graphs $f: X \to Y$ is a pair of functions ${E(X) \to E(Y)}$ and ${V(X) \to V(Y)}$ such that this structure is preserved, \emph{i.e.} $f(\iota(e)) = \iota(f(e))$ and $\overline{f(e)} = f(\overline e)$. This determines a category of finite graphs in which we may consider pullbacks, pushouts, products, and so on.

The pullback in particular is important (\cref{theorem:pullback is intersection} below). This always exists, and can be constructed explicitly. Given maps of finite graphs

\begin{center}
\begin{tikzcd}
& Y \arrow{d}{f} \\
Z \arrow{r}{g} & X \\
\end{tikzcd}
\end{center}
the pullback is given by
\begin{equation*}
V(Y \times_X Z) = \{(v, v') \in V(Y) \times V(Z) : f(v) = g(v')\}
\end{equation*}
\begin{equation*}
E(Y \times_X Z) = \{(e, e') \in E(Y) \times E(Z) : f(e) = g(e')\} \\
\end{equation*}

\subsection{Immersions and coverings}

We define the \emph{star of $v$ in $X$} to be the set 
\begin{equation*}
\text{St}_X(v) = \{e \in E(X) : \iota(e) = v\}
\end{equation*}
Given a map of graphs $f: X \to Y$ and a vertex $v \in V(X)$, we get an induced map $\text{St}_X(v) \to \text{St}_Y(f(v)); e \mapsto f(e)$. A map of graphs is said to be an \emph{immersion} if this induced map is injective for every $v \in V(X)$, and a \emph{covering} if it is bijective for every $v \in V(X)$. (We often denote immersions as $f: X \looparrowright Y$.)

Any map of based graphs $f: (X, x_0) \to (Y, y_0)$ induces a homomorphism of fundamental groups $f_*: \pi_1(X, x_0) \to \pi_1(Y, f(x_0))$ for any $x_0 \in V(X)$. Furthermore, if $f$ is an immersion then this homomorphism is in fact injective. Using the technique of Stallings folding we can represent any finitely generated subgroup $H \leq \pi_1(Y, y_0)$  ($Y$ a finite graph) as a based immersion ${f: (X, x_0) \looparrowright (Y, y_0)}$, where $f_* \pi_1(X, x_0) = H$ and $X$ is a finite graph. For more details, see section 5.4 of \cite{Sta}.

\begin{remark}
Note that $r(H) - 1 = -\chi(X)$ (where $\pi_1(X, x_0) \cong H$). This is because if we pick a maximal spanning tree of $X$, then each edge of $X$ not in this tree determines a basis element of $\pi_1(X, x_0)$ and reduces the Euler characteristic of the graph by 1, and the Euler characteristic of the maximal spanning tree itself is 1. Therefore questions about reduced rank can be reduced to questions about Euler characteristic, which is the technique used in the proofs below.
\end{remark}

\subsection{Useful results about subgroups and graphs}
A key tool needed to investigate the IEHNC using graphs is the following theorem from \cite{Sta}:

\begin{samepage}
\begin{theorem}
\label{theorem:pullback is intersection}
Let
\begin{center}
\begin{tikzcd}
Y_3 \arrow{r}{g_1} \arrow{d}{g_2}
& Y_1 \arrow{d}{f_1}\\
Y_2 \arrow{r}{f_2} & X
\end{tikzcd}
\end{center}
be a pullback diagram of finite graphs. Suppose that $f_1$, $f_2$ are immersions. Let $v_1 \in V(Y_1)$, $v_2 \in V(Y_2), w \in V(X)$ such that
\begin{equation*}
f_1(v_1) = f_2(v_2) = w
\end{equation*}
Let $v_3 = (v_1, v_2) \in V(Y_3)$. Define $f_3 = f_1g_1 = f_2 g_2$. Let 
\begin{equation*}
H_i = (f_i)_* (\pi_1(Y_i, v_i))
\end{equation*}
Then $H_3 = H_1 \cap H_2$.
\end{theorem}
\begin{proof}
See \cite{Sta}, theorem 5.5.
\end{proof}
\end{samepage}

This means that we can explicitly construct the immersion representing the intersection of any two subgroups.

Finally, when the subgroup is of finite index we get extra information about the immersion representing it:

\begin{lemma}
\label{lemma:covering iff finite index}
Let $H \leq F$ be represented by an immersion of finite graphs $f: Y \to X$, where $X$ is a rose (\emph{i.e.} has only one vertex). Then $H$ is of finite index in $F$ iff $Y$ is a covering space of $X$.
\end{lemma}

\begin{proof}
See \cite{Sta}, remark 7.6 on page 562
\end{proof}

\begin{remark}
Note that this in particular implies that the number of vertices in $Y$ is equal to the index of $H$ in $F$. This is an instance of a more general result from the theory of covering spaces, which will be used again below: if $f: (Y, y_0) \to (X, x_0)$ is a based covering map then the index of $f_*\pi_1(Y, y_0)$ in $\pi_1(X, x_0)$ is equal to the number of sheets in the covering.
\end{remark}

\section{Inclusion/Exclusion Hanna Neumann Conjecture}

In order to prove \cref{theorem:IEHNC}, we will make use of the following result:

\begin{lemma}
\label{lemma:index rank relation}
Let $H$ have finite index in $F$. Then
\begin{equation*}
\overline r(H) = \overline r(F) |F: H|
\end{equation*}

Additionally, if $H$ is represented by a covering $g: Y \to Z$ where $\pi(Z, z_0) \cong F$, then
\begin{equation*}
\overline r(H) = \overline r(F) \frac{|V(Y)|}{|V(Z)|}
\end{equation*}
\end{lemma}

\begin{proof}
By picking a free basis for $F$, we can represent $F$ by a rose with $\text{rk}(F)$ petals, say $X$, and $H$ by an immersion of finite graphs $f: W \to X$.

By \cref{lemma:covering iff finite index}, $f$ is a covering and $W$ has $|F:H|$ vertices. $W$ has $\text{rk}(F) |F:H|$ edges (each vertex of $W$ has valence $2 \, \text{rk}(F)$ since $W$ is a covering), so 
\begin{equation*}
\chi(W) = (1 -\text{rk}(F))|F:H|
\end{equation*}
Since $F$ is non-trivial, $\overline r(F) = \text{rk}(F) - 1$ and $\overline r(H) = -\chi(Y)$, which gives the first equality.

To get the second equality, we note that $|F:H| = |V(Y)|/|V(Z)|$ (since both sides are equal to the number of sheets of the covering ${g: Y \to Z}$).
\end{proof}

We are now in a position to prove \cref{theorem:IEHNC}:

\iehnc*

\begin{proof}
Firstly, we note that we can take $F$ to be $H \vee K$:  we have ${K \leq (H\vee K) \leq F}$ and so $|F:K| = |F:H\vee K| |H \vee K: K|$, hence $K$ is of finite index in $H \vee K$.

Identify $F$ with $\pi_1(X, x_0)$, where $X$ is a rose with $\text{rk}(F)$ petals. Let $H$ (resp. $K$) be represented by the immersion of finite graphs $g_1: (Y, y_0) \to (X, x_0)$ (resp. ${g_2: (Z, z_0) \to (X, x_0)}$), and construct the pullback $Y \times_X Z$. Let $W$ be the component of the pullback that contains the vertex $w_0 := (y_0, z_0)$. Then, by \cref{theorem:pullback is intersection}, $\pi_1(W, w_0) \cong H \cap K$.

\begin{center}
\begin{tikzcd}
(W, w_0) \arrow[two heads]{r}{f_1} \arrow{d}{f_2}
& (Y, y_0) \arrow{d}{g_1}\\
(Z, z_0) \arrow[two heads]{r}{g_2} &(X, x_0)
\end{tikzcd}\hspace{40pt}
\begin{tikzcd}
H \cap K \arrow{r}{(f_1)_*} \arrow{d}{(f_2)_*}
& H \arrow{d}{(g_1)_*}\\
K \arrow{r}{(g_2)_*} & F
\end{tikzcd}
\end{center}

$(W, w_0)$ is a based covering space of $(Y, y_0)$, by construction of the pullback. Indeed, given any vertex $(y, z) \in W$, we have a bijection $\text{St}_Z(z) \to \text{St}_X(x_0)$ and hence for each edge $e$ in $\text{St}_Y(y)$ we have exactly one edge in $\text{St}_Z(z)$ whose image in $\text{St}_X(x_0)$ is the same as the image of $e$.

By \cref{lemma:index rank relation} applied to $f_1: W \to Y$ we have that 

\begin{equation*}
\overline r(H \cap K) = \overline r(H) \frac{|V(W)|}{|V(Y)|}
\end{equation*}

Since $W$ is a subgraph of the product $Y \times Z$, we also have

\begin{equation*}
|V(W)| \leq |V(Y)| \, |V(Z)|
\end{equation*}
and so 
\begin{equation*}
\overline r(H \cap K) \leq \overline r(H) |V(Z)|
\end{equation*}

Finally, we can use \cref{lemma:index rank relation} on $g_2: Z \to X$ to obtain
\begin{equation*}
\bar{r}(K) = \bar r(F) \; |V(Z)|
\end{equation*}
and so
\begin{equation*}
\overline r(H \cap K) \, \overline r(H \vee K) \leq \overline r (H) \, \overline r(K)
\end{equation*}
as desired.
\end{proof}

\section{Counterexamples}
\label{section:counterexamples}

The IEHNC does not necessarily hold when neither $H$ nor $K$ is of finite index in $H \vee K$. Indeed, we can show that there is no $c > 0$ such that
\begin{equation*}
\overline r(H \cap K) \, \overline r(H \vee K) \leq c \, \overline r(H) \, \overline r (K)
\end{equation*}
holds for all $H, K$ finitely generated subgroups of $F$. An example demonstrating this is given below.

\begin{example}
\label{example:IEHNC counterexample}
\begin{figure}
	\begin{tikzpicture}[node distance=1.3cm,>=stealth',bend angle=45,auto]

  \tikzstyle{vertex}=[circle,fill=black,inner sep=0.5mm]  
  \tikzstyle{rootnode}=[circle,draw=black,inner sep=0.1mm]  

  \tikzstyle{place}=[circle,thick,draw=blue!75,fill=blue!20,minimum size=6mm]

  \tikzstyle{red place}=[place,draw=red!75,fill=red!20]
  \tikzstyle{transition}=[rectangle,thick,draw=black!75,
  			  fill=black!20,minimum size=4mm]

  \tikzstyle{every label}=[red]

  \begin{scope}
    \node at (-1.5, 0) {$Y \times_X Z$};

	\node at (1.5, 0.6) {$v$ vertices};
	\draw [decorate, decoration={brace}] (0, 0.3) -- (3, 0.3);        
    
    \node at (0,0) [rootnode] (root) {$\ast$};
    \node at (1,0) [vertex] (n1)   {};

    \draw [->, bend left, color=red] (root) to (n1);
    \draw [->, bend right, color=blue] (root) to (n1);

    \node at (2,0)          (gap)  {...};

    \draw [->, bend left, color=red] (n1) to (gap);
    \draw [->, bend right, color=blue] (n1) to (gap);

    \node at (3,0) [vertex] (n2)   {};
    
    \draw [->, bend left, color=red] (gap) to (n2);
    \draw [->, bend right, color=blue] (gap) to (n2);
  \end{scope}

  \begin{scope}[xshift=6.5cm]
    \node at (0, 0) [rootnode] (root) {$\ast$};
    
    \node at (-0.85, 0.6) {$\ell$ loops};
	\draw [decorate, decoration={brace}] (-1.6, 0.3) -- (-0.1, 0.3);    
    \draw [->, loop left, draw=pink1, min distance=5mm] (root) to (root);
    \draw [->, loop left, draw=pink2, min distance=10mm] (root) to (root);
    \draw [->, loop left, draw=pink3,min distance=20mm] (root) to (root);
    \node at (-1.2, 0) {$\ldots$};
    
    \node at (1,0) [vertex] (n1)   {};

    \draw [->, bend left, color=red] (root) to (n1);
    \draw [->, bend right, color=blue] (root) to (n1);

    \node at (2,0)          (gap)  {$\ldots$};

    \draw [->, bend left, color=red] (n1) to (gap);
    \draw [->, bend right, color=blue] (n1) to (gap);

    \node at (3,0) [vertex] (n2)   {};
    
    \draw [->, bend left, color=red] (gap) to (n2);
    \draw [->, bend right, color=blue] (gap) to (n2);
    
    \node at (4,0) {$Y$};
  \end{scope}
  
  \begin{scope}[yshift=-2cm]
  	\node at (0,0) [rootnode] (root) {$\ast$};
  	\draw [->, loop right, color=red] (root) to (root);
  	\draw [->, loop below, color=blue] (root) to (root);
  	\node at (-1.5, 0) {$Z$};
  \end{scope}
  
  \begin{scope}[xshift=6.5cm, yshift=-2cm]
  	\node at (0,0) [rootnode] (root) {$\ast$};
  	\draw [->, loop left, draw=pink1, min distance=5mm] (root) to (root);
    \draw [->, loop left, draw=pink2, min distance=10mm] (root) to (root);
    \draw [->, loop left, draw=pink3,min distance=20mm] (root) to (root);
    \node at (-1.2, 0) {$\ldots$};
 	\draw [->, loop right, color=red] (root) to (root);
  	\draw [->, loop below, color=blue] (root) to (root);
  	\node at (1.5, 0) {$X$};
  \end{scope}

  \draw [->] (3.5,0) -- (4.5,0);
  \draw [->] (0,-0.5) -- (0,-1.5);
  \draw [->] (3.5,-2) -- (4.5,-2);
  \draw [->] (6.5,-0.5) -- (6.5,-1.5);

\end{tikzpicture}
	\caption{Counterexample to the IEHNC}
	\label{figure:IEHNC counterexample}
\end{figure}
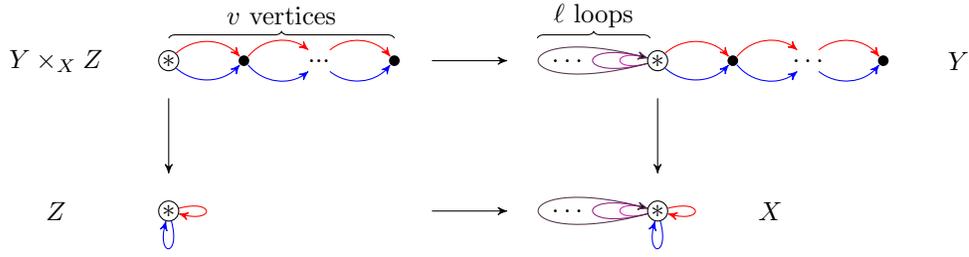

Let
\begin{equation*}
H = \langle x_2 x_1^{-1}, x_1 x_2 x_1^{-2}, \ldots, x_1^{v-2}x_2x_1^{-(v-1)}, x_3, x_4, \ldots, x_{\ell + 2} \rangle
\end{equation*}
and
\begin{equation*}
K = \langle x_1, x_2 \rangle 
\end{equation*}
We then get
\begin{gather*}
H \cap K = \langle x_2 x_1^{-1}, x_1 x_2 x_1^{-2}, \ldots, x_1^{v-2}x_2x_1^{-(v-1)} \rangle \\
H \vee K = \langle x_1, \ldots, x_{\ell + 2} \rangle
\end{gather*}
These are illustrated in the pushout diagram shown in \cref{figure:IEHNC counterexample}, in which ${H \cong \pi_1(Y, y_0)}$ and $K \cong \pi_1(Z, z_0)$. For this choice of $H$ and $K$ we have
\begin{equation*}
\frac{\overline r(H \cap K) \, \overline r(H \vee K)}{\overline r(H) \, \overline r(K)} = \frac{(v-2)(\ell+1)}{v+\ell-2}
\end{equation*}

Setting $v = \ell$, we obtain $(v^2 - v-2)/(2v-2)$, so as $v \to \infty$ the ratio gets arbitrarily large.
\end{example}
\vspace{5mm}

Using the graph generation algorithm detailed in \cref{section:graph generation}, we were able to find a counterexample to Guzman's ``group-theoretic conjecture'' as well.

\begin{example}
\label{example:guzman counterexample}

Let $F = F(a, b, c, d, x, y)$, and let
\begin{gather*}
H = \langle a, b, x, y^2, yxy^{-1} \rangle \\
K = \langle c, d, y, x^2, xyx^{-1} \rangle
\end{gather*}

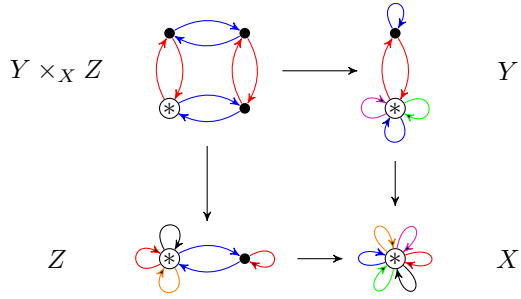
\begin{figure}
	\begin{center}
	\begin{tikzpicture}[node distance=1.3cm,>=stealth',bend angle=45,auto]

  \tikzstyle{vertex}=[circle,fill=black,inner sep=0.5mm]  
  \tikzstyle{rootnode}=[circle,draw=black,inner sep=0.1mm]  

  \tikzstyle{place}=[circle,thick,draw=blue!75,fill=blue!20,minimum size=6mm]

  \tikzstyle{red place}=[place,draw=red!75,fill=red!20]
  \tikzstyle{transition}=[rectangle,thick,draw=black!75,
  			  fill=black!20,minimum size=4mm]

  \tikzstyle{every label}=[red]

  \begin{scope}
      \node at (-1.5, 0.5) {$Y \times_X Z$};
    
      \node at (0,0) [rootnode] (bl) {$\ast$};
      \node at (1,0) [vertex] (br)   {};
      \node at (0,1) [vertex] (tl) {};
      \node at (1,1) [vertex] (tr) {};

      \draw [->, bend left=30,  color=blue] (tr) to (tl);
      \draw [<-, bend right=30, color=blue] (tr) to (tl);
      \draw [->, bend left=30, color=red] (tl) to (bl);
      \draw [<-, bend right=30, color=red] (tl) to (bl);
      \draw [->, bend left=30,  color=blue] (br) to (bl);
      \draw [<-, bend right=30, color=blue] (br) to (bl);
      \draw [->, bend left=30, color=red] (br) to (tr);
      \draw [<-, bend right=30, color=red] (br) to (tr);
  \end{scope}

  \begin{scope}[xshift=3cm]
    \node at (0, 0) [rootnode] (b) {$\ast$};
    \node at (0, 1) [vertex] (t) {};
    
    \draw [->, loop left, draw=pink1, distance=5mm, out=210, in=150] (b) to (b);
    \draw [->, loop below, draw=blue, distance=5mm, out=-60, in=-120] (b) to (b);
    \draw [->, loop right, draw=green, distance=5mm, out=30, in=-30] (b) to (b);
    
    \draw [->, bend left=30, color=red] (b) to (t);
    \draw [<-, bend right=30, color=red] (b) to (t);
    
    \draw [->, loop above, draw=blue, distance=5mm, out=120, in=60] (t) to (t);
    
    \node at (1.5,0.5) {$Y$};
  \end{scope}
  
  \begin{scope}[yshift=-2cm]
    \node at (0, 0) [rootnode] (l) {$\ast$};
    \node at (1, 0) [vertex] (r) {};
    
    \draw [->, loop left, draw=red, distance=5mm, out=210, in=150] (l) to (l);
    \draw [->, loop below, draw=orange, distance=5mm, out=-60, in=-120] (l) to (l);
    \draw [->, loop above, draw=black, distance=5mm, out=120, in=60] (l) to (l);
    
    \draw [->, bend left=30, color=blue] (l) to (r);
    \draw [<-, bend right=30, color=blue] (l) to (r);
    
    \draw [->, loop right, draw=red, distance=5mm, out=30, in=-30] (r) to (r);
    \node at (-1.5,0) {$Z$};    
  \end{scope}
  
  \begin{scope}[xshift=3cm, yshift=-2cm]
  	\node at (0,0) [rootnode] (root) {$\ast$};
  	\draw [->, loop, draw=red, out=20, in=-20, distance=5mm] (root) to (root);
  	\draw [->, loop, draw=pink1, out=80, in=40, distance=5mm] (root) to (root);
  	\draw [->, loop, draw=orange, out=140, in=100, distance=5mm] (root) to (root);
 	\draw [->, loop, draw=blue, out=200, in=160, distance=5mm] (root) to (root);
  	\draw [->, loop, draw=green, out=260, in=220, distance=5mm] (root) to (root);
  	\draw [->, loop, draw=black, out=320, in=280, distance=5mm] (root) to (root);

  	\node at (1.5, 0) {$X$};
  \end{scope}

  \draw [->] (1.5,0.5) -- (2.5,0.5);
  \draw [->] (0.5,-0.5) -- (0.5,-1.5);
  \draw [->] (1.7,-2) -- (2.3,-2);
  \draw [->] (3,-0.7) -- (3,-1.3);

\end{tikzpicture}
	\end{center}
	\caption{Counterexample to conjecture \ref{conjecture:guzman} (Guzman)}
	\label{figure:guzman counterexample}
\end{figure}

We then get
\begin{gather*}
H \cap K = \langle y^2, yx^2y^{-1}, x^2, yxy^{-1}x, yxyx \rangle \\
H \vee K = \langle a, b, c, d, x, y \rangle
\end{gather*}
and so disprove Guzman's conjecture.

These are illustrated in the pushout diagram shown in \cref{figure:guzman counterexample} (where red and blue correspond to $x$ and $y$).
\end{example}

\section{Graph-based generation algorithm}
\label{section:graph generation}

In order to investigate the above questions about subgroups of free groups, and to find the first counterexamples (though not the ones presented above), we used GAP to generate random subgroups of free groups.

Historically, random subgroups of free groups were usually generated by the ``word-based distribution'', in which $k$-tuples of reduced words $(g_1, \ldots, g_k)$ are chosen in $F$, with each $g_i$ having length less than some fixed $n$. We then consider the subgroup $H = \langle g_1, \ldots, g_k \rangle$.

Recently, Bassino, Nicaud and Weil proposed the ``graph-based distribution'', which generates a random Stallings graph with a fixed number of vertices, and then computes its fundamental group (see section 3 of \cite{bass}). The algorithm first generates a random $r(F)$-tuple of partial injections (by a procedure given explicitly in \cite{bass}), with an $a$-labelled edge going from $v$ to $w$ in the Stallings graph if and only if the partial injection corresponding to $a$ sends $v$ to $w$. If the generated graph is not connected, or has leaves other than the basepoint, then it is discarded and a new one is generated (a ``rejection algorithm''). The fundamental group of this graph can be found by constructing a spanning tree, and then the GAP package FGA was used to calculate the intersection, join and rank of the subgroups generated.

Both the word- and graph-based distributions were able to generate counterexamples; however, the proportion of examples checked that were counterexamples is significantly higher in the graph-based distribution (\cref{figure:IEHNC counterexample frequencies}), as was the proportion of examples which had a non-trivial intersection (\cref{figure:IEHNC trivial frequencies}). Interestingly, the dip in proportion of non-trivial intersections in the graph-based distribution coincides with the peak of the proportion of counterexamples generated, contrary to what might be na\"ively expected; we are unable to explain this.

These figures suggest that the graph-based distribution is better-suited for investigating similar questions using computers.

\begin{figure}
	\begin{center}
	\includegraphics[width=35em]{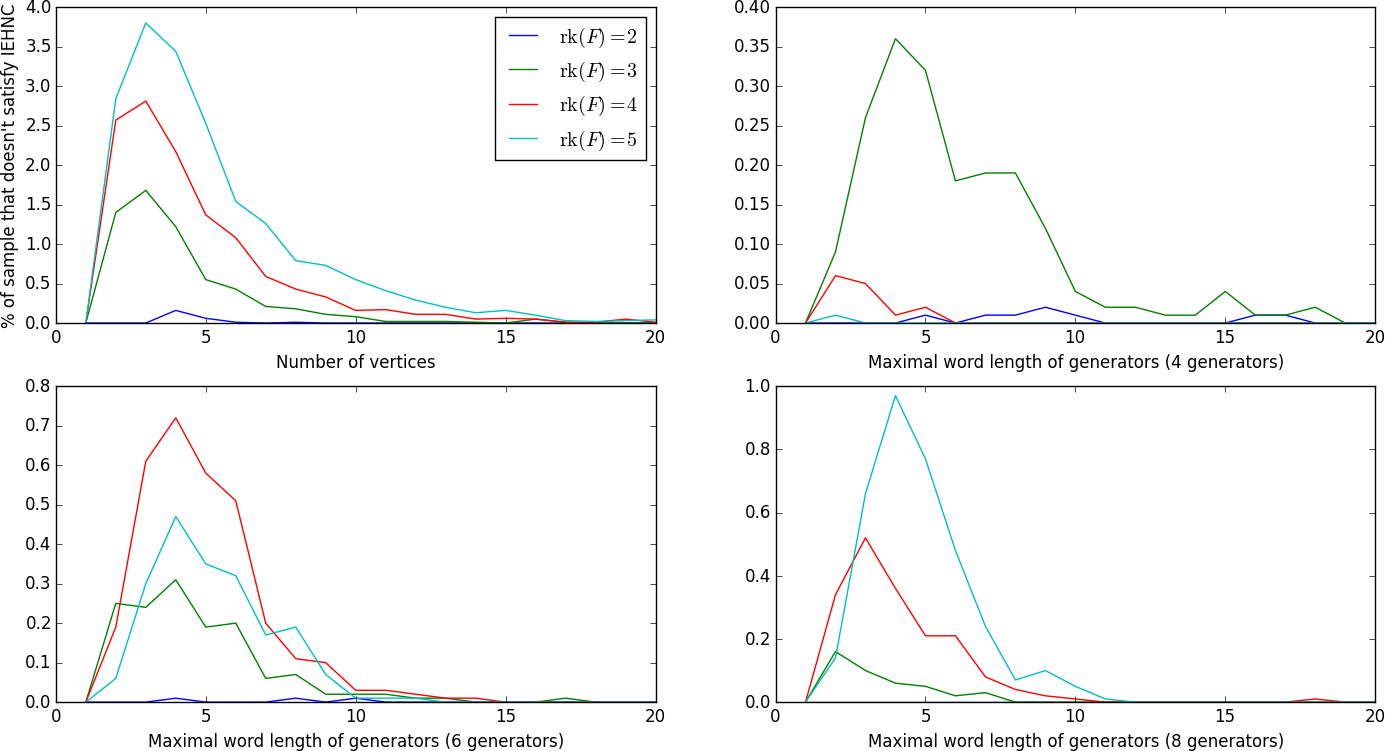}
	\end{center}
	\caption{The percentage of pairs of subgroups out of a sample of 10,000 randomly-generated subgroup pairs that fail to satisfy the IEHNC, against the parameter to the model (either vertices or maximal length of generating word). The sample contained only subgroups whose reduced ranks are strictly positive, as otherwise the IEHNC holds trivially. The subplots are (left-to-right, top-to-bottom): graph-based, word-based with 4 generators, word-based with 6 generators, and word-based with 8 generators. The different trendlines represent different ranks of the ambient group, as indicated in the legend.}
	\label{figure:IEHNC counterexample frequencies}
\end{figure}

\begin{figure}
	\begin{center}
	\includegraphics[width=35em]{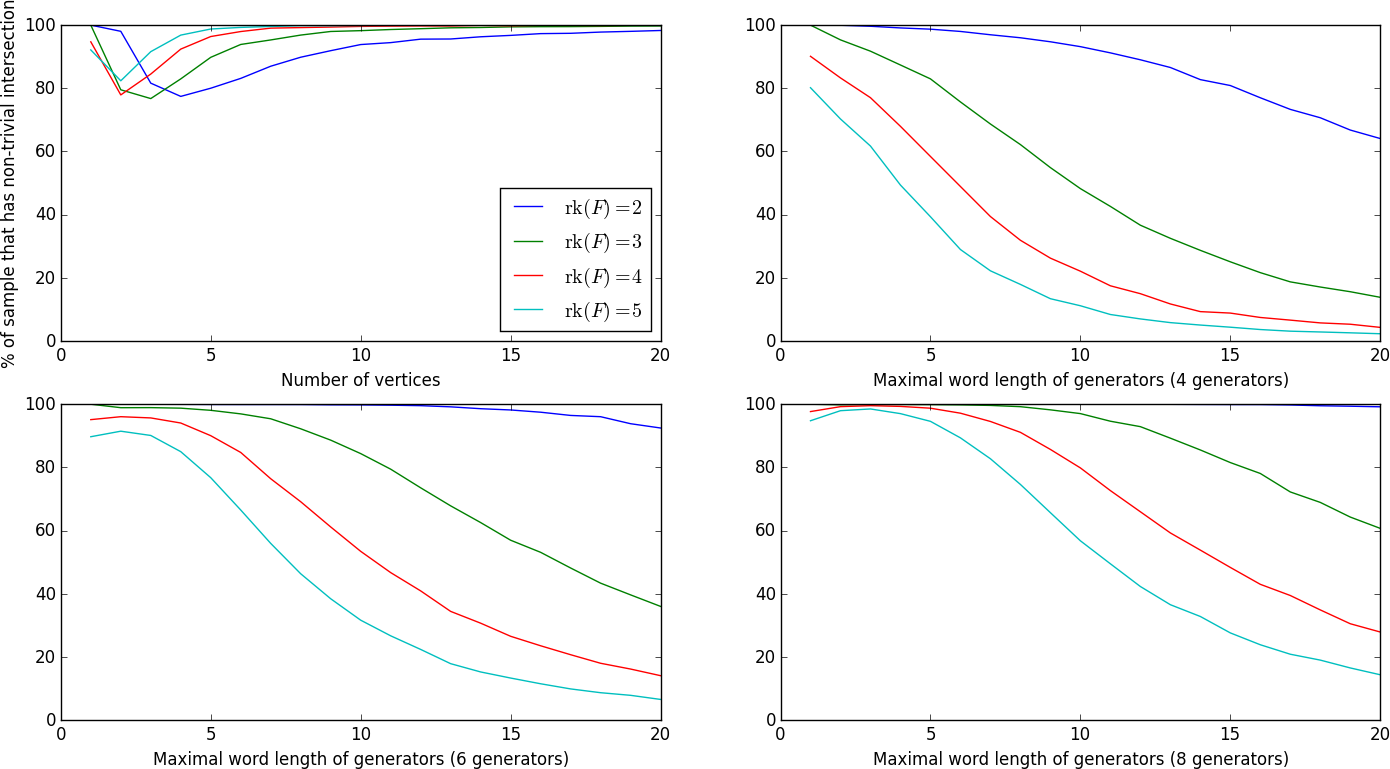}
	\end{center}
	\caption{The percentage of pairs of subgroups out of a sample of 10,000 randomly-generated subgroup pairs that have a non-trivial intersection, against the parameter to the model (either vertices or maximal length of generating word). The sample used was the same as for \cref{figure:IEHNC counterexample frequencies}. The subplots and trendlines are also the same as in that figure.}
	\label{figure:IEHNC trivial frequencies}
\end{figure}

\end{document}